\documentclass[11pt]{article}

\usepackage[T1]{fontenc}
\usepackage{lmodern}
\usepackage[a4paper,margin=31mm]{geometry}
\usepackage{amsmath,amssymb,amsthm,mathtools}
\usepackage{microtype}
\usepackage{enumitem}
\usepackage{xcolor}
\usepackage{tikz}
\usetikzlibrary{arrows.meta}
\usepackage[colorlinks=true,linkcolor=blue,citecolor=blue,urlcolor=blue]{hyperref}

\definecolor{revisionblue}{RGB}{0,92,170}
\newcommand{\rev}[1]{\textcolor{revisionblue}{#1}}

\numberwithin{equation}{section}

\newtheorem{theorem}{Theorem}[section]
\newtheorem{proposition}[theorem]{Proposition}
\newtheorem{lemma}[theorem]{Lemma}
\newtheorem{corollary}[theorem]{Corollary}
\theoremstyle{remark}
\newtheorem{remark}[theorem]{Remark}
\newtheorem*{definition}{Definition}

\newcommand{\C}{\mathbb{C}}
\newcommand{\R}{\mathbb{R}}

\newcommand{\dist}{\operatorname{dist}}
\newcommand{\diam}{\operatorname{diam}}
\newcommand{\dimA}{\operatorname{dim}_{\mathrm A}}
\newcommand{\dimH}{\operatorname{dim}_{\mathrm H}}
\newcommand{\dimM}{\overline{\operatorname{dim}}_{\mathrm M}}
\newcommand{\dimMi}{\overline{\operatorname{dim}}_{\mathrm M,\Omega}}

\title{Distance-Weighted Norm Equivalences\\
for Analytic Functions on John Domains}
\author{Katsuhiko Matsuzaki\thanks{Department of Mathematics, School of Education, Waseda University,
Nishi-Waseda 1-6-1, Shinjuku, Tokyo 169-8050, Japan.\
\texttt{matsuzak@waseda.jp}}
\ and Huaying Wei\thanks{Center for Applied Mathematics, Tianjin
University, No.~92 Weijin Road, Tianjin 300072, PR China.
\texttt{hywei@tju.edu.cn}}}
\date{}

\begin{document}

\maketitle
\insert\footins{\noindent\footnotesize
The first author is partially supported  by  Japan Society for the
Promotion of Science (KAKENHI 23K25775 and 23K17656), and the second
author is partially supported  by the National Natural Science Foundation
of China (grant nos.~12271218 and 12571083).}

\begin{abstract}
Let $\Omega\subset\C$ be a bounded John domain and set
$\delta(z)=\dist(z,\partial\Omega)$.  For $1<p<\infty$ and
$\alpha>\dimA(\partial\Omega)-2$, we establish a norm equivalence between
$\int_\Omega |g|^p\delta^\alpha\,dA$ and
$\int_\Omega |g'|^p\delta^{\alpha+p}\,dA$ for analytic functions $g$ on
$\Omega$, with a point-evaluation term fixing the additive constant.
The estimate of the derivative term is local and holds on every proper
planar domain, whereas the converse follows from a distance-weighted
Poincar\'e inequality on John domains.

Taking $\alpha=mp-2$ yields the corresponding comparison between the
$m$-th and $(m+1)$-st derivatives.  For $m\ge2$ the boundary-dimension
condition is automatic, so the only dimension-sensitive case is the
comparison between $\int_\Omega |f'|^p\delta^{p-2}\,dA$ and
$\int_\Omega |f''|^p\delta^{2p-2}\,dA$ when $1<p<2$.  We show that this
restriction is sharp within the class of quasidisks by using self-similar
Rohde snowflakes.  We also construct, for every $s>1$, an inward-cusp
$s$-John domain on which the comparison fails, showing that the ordinary
John condition cannot in general be weakened.
\end{abstract}

\medskip
\noindent
\textbf{Keywords.}
John domain, distance weight, weighted Poincar\'e inequality, Ahlfors
regularity, Assouad dimension, Minkowski dimension, snowflake.

\medskip
\noindent
\textbf{2020 Mathematics Subject Classification.}
30H20, 30C62, 46E35, 26D10.

\section{Introduction and main results}

The starting point of the present work is a classical higher-order
derivative characterization of analytic Besov spaces on the unit
disk due to Zhu; see 
\cite[Theorem~A and p.~327]{Zhu-analytic-Besov}.
Let
\[
 d\lambda(z)=\frac{dA(z)}{(1-|z|^2)^2}
\]
be the M\"obius invariant measure on the unit disk $\mathbb D$,
where $dA(z)=dx\,dy$ is Euclidean area measure.  Its normalization is
immaterial for the norm equivalences considered here.

For
$1<p<\infty$, the analytic Besov seminorm is given by
\[
 \|f\|_{B_p}
 =
 \bigl\|(1-|z|^2)f'(z)\bigr\|_{L^p(\mathbb D,d\lambda)}.
\]
Zhu proved that, for every integer $m\geq2$, this seminorm can
equivalently be described in terms of the $m$-th derivative.  More
precisely,
\[
 \|f\|_{B_p}
 \asymp
 |f'(0)|+\cdots+|f^{(m-1)}(0)|  
 +\bigl\|(1-|z|^2)^m f^{(m)}(z)
   \bigr\|_{L^p(\mathbb D,d\lambda)}.
\]
Since $1-|z|^2\asymp \delta_{\mathbb D}(z)
 :=\operatorname{dist}(z,\partial\mathbb D)$,
this seminorm equivalence may be rewritten, in terms of powers of the
boundary distance, as
\[
 \left(
   \int_{\mathbb D}|f'(z)|^p
   \delta_{\mathbb D}(z)^{p-2}\,dA(z)
  \right)^{1/p} \asymp
 \sum_{j=1}^{m-1}|f^{(j)}(0)|+
 \left(
   \int_{\mathbb D}|f^{(m)}(z)|^p
   \delta_{\mathbb D}(z)^{mp-2}\,dA(z)
  \right)^{1/p}.
\]

The purpose of this paper is to develop a domain-geometric extension
of this disk result.  We replace $\mathbb D$ by a bounded John domain
$\Omega$ and the radial factor $1-|z|^2$ by the intrinsic boundary
distance
$\delta(z)=\delta_\Omega(z):=\dist(z,\partial\Omega)$.  Our original
motivation was to compare
\[
 \int_\Omega |f'(z)|^p\delta(z)^{p-2}\,dA(z)
 \quad\text{and}\quad
 \int_\Omega |f''(z)|^p\delta(z)^{2p-2}\,dA(z).
\]
The two integrals above are the first two members of the natural derivative
scale
\begin{equation}\label{eq:natural-scale}
 \int_\Omega |f^{(m)}(z)|^p\delta(z)^{mp-2}\,dA(z),
 \qquad m=1,2,\ldots.
\end{equation}
The proofs do not depend on the particular exponent $mp-2$.  We therefore first establish  a seminorm equivalence for the general distance weight
$\delta^\alpha$ and then specialize to $\alpha=mp-2$.

The estimate of the derivative by the function itself is entirely
local and requires neither a Whitney decomposition nor any global
geometric hypothesis.  The converse estimate is global and follows
from a weighted improved Poincar\'e inequality on John domains.

Recall that a bounded domain $\Omega\subset\mathbb C$ is an
\emph{$s$-John domain}, where $s\ge1$, if there exist a point
$x_0\in\Omega$, called a John center, and a constant $C_J>0$ such that
every $x\in\Omega$ can be joined to $x_0$ by a rectifiable curve
$\gamma:[0,\ell]\to\Omega$, parametrized by arclength from $x$, for which
\[
 \delta(\gamma(t))\ge C_Jt^s,
 \qquad 0\le t\le\ell.
\]
The case $s=1$ is the ordinary John condition.  Because the domain is
bounded, an ordinary John curve has uniformly bounded length; hence every
John domain is an $s$-John domain for $s>1$ after the constant $C_J$ is
adjusted.  The converse need not hold.

\begin{definition}
Let $(X,d)$ be a metric space and let $E\subset X$. 
For $x\in X$ and $R>0$, denote by $B(x,R)$
the open ball of radius $R$ centered at $x$.
The \emph{Assouad dimension} of $E$, denoted by
$\dim_{\mathrm A}(E)$, is defined by
the infimum of all exponents
$s\geq0$ for which there exists a constant $C>0$, independent of
$x$, $r$, and $R$, such that
\[
N_r\bigl(E\cap B(x,R)\bigr)
\leq
C\left(\frac{R}{r}\right)^s
\]
for all $x\in E$ and all $0<r<R$. Here $N_r(A)$ denotes the smallest
number of open balls of radius $r$ needed to cover a set $A\subset X$.
\end{definition}

For background on these and related metric notions, see \cite{H}. 
A nonempty bounded set $E\subset\mathbb C$ is called \emph{$d$-Ahlfors regular} for $d>0$ if there is a
constant $A\ge1$ such that
\[
A^{-1}r^d\le \mathcal H^d(E\cap B(\xi,r))\le A r^d
\]
for every $\xi\in E$ and $0<r\le\diam(E)$,
where $\mathcal H^d$ denotes the $d$-dimensional Hausdorff measure.
In general, the Hausdorff dimension $\dimH(E)$ is always less than or equal to $\dimA(E)$.
If $E$ is $d$-Ahlfors regular, then $\dimH(E)=\dimA(E)=d$.

Our general result is as follows.
\begin{theorem}\label{thm:main-general}
Let $\Omega\subset\C$ be a bounded John domain, let $z_0\in\Omega$,
and let $1<p<\infty$.  Suppose that
\begin{equation}\label{eq:general-assumptions}
 \alpha>\dim_{\mathrm A}(\partial\Omega)-2.
\end{equation}
Then every analytic function $g$ in $\Omega$ satisfies
\begin{equation}\label{eq:main-equivalence}
 \int_\Omega |g(z)|^p\delta(z)^\alpha\,dA(z)
 \asymp
 |g(z_0)|^p\delta(z_0)^{\alpha+2}
 +
 \int_\Omega |g'(z)|^p\delta(z)^{\alpha+p}\,dA(z),
\end{equation}
where the comparison constants may depend on
$\Omega,p,\alpha$, and $z_0$, but not on $g$.
\end{theorem}

Taking $g=f^{(m)}$ and $\alpha=mp-2$ gives the natural consecutive
higher-derivative comparison.

\begin{corollary}
\label{cor:consecutive}
Let $\Omega\subset\C$ be a bounded John domain, let $z_0\in\Omega$,
let $1<p<\infty$, and let $m\ge1$ be an integer.  Assume that
$mp>\dimA(\partial\Omega)$.
Then every analytic function $f$ in $\Omega$ satisfies
\begin{align}\label{eq:consecutive}
 \int_\Omega |f^{(m)}(z)|^p\delta(z)^{mp-2}\,dA(z)
 \asymp{}& |f^{(m)}(z_0)|^p\delta(z_0)^{mp} \\
 &+\int_\Omega |f^{(m+1)}(z)|^p
       \delta(z)^{(m+1)p-2}\,dA(z).
 \nonumber
\end{align} 
\end{corollary}

For the boundary of a bounded planar domain,
$1\le\dimA(\partial\Omega)\le2$, while the John condition excludes the
endpoint $2$.  Indeed, every bounded John domain satisfies a uniform
interior corkscrew condition (see \cite[Lemma~8.9]{ABBS}); consequently,
$\partial\Omega$ is uniformly porous in $\mathbb C$.  By
\cite[Theorem~5.2]{Luukkainen}, a set $E\subset\mathbb R^n$ is uniformly
porous if and only if $\dimA(E)<n$.  It follows that
$\dimA(\partial\Omega)<2$.  Combining this observation with
Corollary~\ref{cor:consecutive} gives the following consequences.

\begin{corollary}
\label{cor:delicate}
Let $\Omega\subset\C$ be a bounded John domain and let $1<p<\infty$.
Then the following assertions hold.
\begin{enumerate}[label=\textup{(\roman*)}]
 \item For every integer $m\ge2$, the comparison
       \eqref{eq:consecutive} holds without any additional assumption
       on $\partial\Omega$.
 \item For $m=1$, the comparison holds whenever $2\le p<\infty$.
 \item For $m=1$ and $1<p<2$, the comparison holds if $\dimA(\partial \Omega)<p$.
\end{enumerate}
Consequently, in the natural scale \eqref{eq:natural-scale}, the only
comparison for which a boundary-dimension restriction can occur is the
one between $f'$ and $f''$, with weights $\delta^{p-2}$ and
$\delta^{2p-2}$.
\end{corollary}

Bounded quasidisks are uniform domains and hence John domains.  A bounded
chord-arc domain is a quasidisk whose boundary is $1$-Ahlfors regular, so
$\dimA(\partial\Omega)=1$.  Corollary~\ref{cor:delicate} therefore shows
that, when $m=1$, the comparison holds on every bounded chord-arc domain
for $1<p<\infty$ and on every bounded quasidisk for $2\le p<\infty$.

The boundary-dimension hypothesis $\dimA(\partial\Omega)<p$ in
part~\textup{(iii)} is sharp even within the class of quasidisks.  In
Section~\ref{sec:counterexamples}, for each $1<p<2$ we construct a
bounded Rohde snowflake quasidisk with $\dimA(\partial\Omega)\ge p$ and
$\delta^{p-2}\notin L^1(\Omega)$.  The constant analytic function
$g\equiv1$ then disproves the global estimate with weights
$\delta^{p-2}$ and $\delta^{2p-2}$.

Part~\textup{(i)} removes the boundary-dimension hypothesis when
$m\ge2$, but the John-domain hypothesis itself cannot in general be
weakened.  A standard geometric enlargement of the class of John domains
is the class of $s$-John domains with $s>1$.  In
Section~\ref{sec:cusp-counterexample}, for every prescribed $s>1$ we
construct an $s$-John inward-cusp domain with rectifiable boundary on
which the global comparison fails for every $m\ge1$.

The consecutive estimates can also be iterated: 
The assumptions for the order $m$ imply the corresponding assumptions
for every $k\ge m$.  Indeed, $kp\ge mp>\dimA(\partial\Omega)$.
We may therefore apply Corollary~\ref{cor:consecutive} successively for
$k=m,m+1,\ldots,N-1$.  A finite iteration yields

\begin{corollary}
\label{cor:iterated}
Under the assumptions of Corollary~\ref{cor:consecutive}, let $N>m$ be
an integer.  Then
\begin{align}\label{eq:iterated}
 \int_\Omega |f^{(m)}(z)|^p\delta(z)^{mp-2}\,dA(z)
 \asymp{}& \sum_{k=m}^{N-1}
 |f^{(k)}(z_0)|^p\delta(z_0)^{kp} \\
 &+\int_\Omega |f^{(N)}(z)|^p\delta(z)^{Np-2}\,dA(z).
 \nonumber
\end{align}
In particular, if $m\ge2$, this holds on every bounded John domain
without further boundary assumptions.
\end{corollary}

The paper is organized as follows.  Sections~\ref{sec:local} and
\ref{sec:base-disk} prove the two directions of
Theorem~\ref{thm:main-general} (see Propositions \ref{prop:local-reverse} and \ref{prop:global-analytic}).  Section~\ref{sec:poincare-input}
collects the required weighted Poincar\'e inequality and the relevant
distance-weight integrability criterion.  Sections~\ref{sec:counterexamples}
and \ref{sec:cusp-counterexample} use Rohde snowflakes and an inward-cusp
domain, respectively, to establish the sharpness of the boundary-dimension
and domain-geometric hypotheses.


\section{\texorpdfstring{The local direction:
from lower to higher derivatives}{The local direction: from lower to
higher derivatives}}
\label{sec:local}

In this section no John condition is needed.  In fact, the argument
works on every proper planar domain and for every real weight exponent $\alpha$.

\begin{lemma}
\label{lem:point}
Let $\Omega\subsetneq\C$ be a domain, $1<p<\infty$, $\alpha\in\R$,
and let $g$ be analytic in $\Omega$.  Then, for every $z_0\in\Omega$,
\begin{equation}\label{eq:point-evaluation}
 |g(z_0)|^p\delta(z_0)^{\alpha+2}
 \le C_{\alpha}\int_\Omega |g(z)|^p\delta(z)^\alpha\,dA(z).
\end{equation}
\end{lemma}

\begin{proof}
Put $\delta_0=\delta(z_0)$ and
$B_0=B(z_0,\delta_0/2)$.  Since $|g|^p$ is subharmonic,
\[
 |g(z_0)|^p
 \le \frac{4}{\pi\delta_0^2}\int_{B_0}|g(z)|^p\,dA(z).
\]
For $z\in B_0$ we have
$\delta_0/2\le\delta(z)\le3\delta_0/2$.  Hence
\[
 \int_{B_0}|g(z)|^p\,dA(z)
 \le C_\alpha\delta_0^{-\alpha}
 \int_{B_0}|g(z)|^p\delta(z)^\alpha\,dA(z).
\]
Combining the two estimates proves \eqref{eq:point-evaluation}.
\end{proof}

\begin{lemma}
\label{lem:cauchy-average}
Let $\Omega\subsetneq\C$ be a domain, $1<p<\infty$, and let $g$ be
analytic in $\Omega$.  For every $z\in\Omega$,
\begin{equation}\label{eq:local-cauchy}
 |g'(z)|^p
 \le C_p\delta(z)^{-p-2}
 \int_{B(z,\delta(z)/4)}|g(\zeta)|^p\,dA(\zeta).
\end{equation}
\end{lemma}

\begin{proof}
Fix $z\in\Omega$ and put $R=\delta(z)/4$.  For every
$\rho\in[R/2,R]$, Cauchy's formula gives
\[
 g'(z)=\frac{1}{2\pi\rho}
 \int_0^{2\pi}g(z+\rho e^{it})e^{-it}\,dt.
\]
By H\"older's inequality,
\[
 |g'(z)|^p
 \le C_p\rho^{-p}\int_0^{2\pi}|g(z+\rho e^{it})|^p\,dt.
\]
Multiplying by $\rho$ and integrating over $R/2\le\rho\le R$, we obtain
\[
 R^2|g'(z)|^p
 \le C_pR^{-p}\int_{B(z,R)}|g(\zeta)|^p\,dA(\zeta).
\]
Since $R=\delta(z)/4$, this is \eqref{eq:local-cauchy}.
\end{proof}

\begin{proposition}
\label{prop:local-reverse}
Let $\Omega\subsetneq\C$ be a domain, let $1<p<\infty$, and let
$\alpha\in\R$.  Every analytic function $g$ in $\Omega$ satisfies
\begin{equation}\label{eq:derivative-from-function}
 \int_\Omega |g'(z)|^p\delta(z)^{\alpha+p}\,dA(z)
 \le C_{p,\alpha}\int_\Omega |g(z)|^p\delta(z)^\alpha\,dA(z).
\end{equation}
Consequently, for every $z_0\in\Omega$,
\begin{align}\label{eq:complete-local-reverse}
 |g(z_0)|^p\delta(z_0)^{\alpha+2}
 +\int_\Omega |g'(z)|^p\delta(z)^{\alpha+p}\,dA(z)
 \le C_{p,\alpha}\int_\Omega |g(z)|^p\delta(z)^\alpha\,dA(z).
\end{align}
\end{proposition}

\begin{proof}
Multiplying \eqref{eq:local-cauchy} by $\delta(z)^{\alpha+p}$ and
integrating gives
\begin{align*}
 \int_\Omega |g'(z)|^p\delta(z)^{\alpha+p}\,dA(z)
 &\le C_p\int_\Omega\delta(z)^{\alpha-2}
 \left(\int_{B(z,\delta(z)/4)}|g(\zeta)|^p\,dA(\zeta)\right)dA(z).
\end{align*}
The integrand is nonnegative, so Tonelli's theorem permits us to
reverse the order of integration.  For fixed $\zeta\in\Omega$, set
\[
 E_\zeta=\left\{z\in\Omega:
 |z-\zeta|<\frac{\delta(z)}{4}\right\}.
\]
If $z\in E_\zeta$, the $1$-Lipschitz property of the distance function
gives
$3\delta(z)/4<\delta(\zeta)<5\delta(z)/4$.
In particular,
$E_\zeta\subset B(\zeta,\delta(\zeta)/3)$ and
$\delta(z)\asymp\delta(\zeta)$ on $E_\zeta$.  Therefore
\[
 \int_{E_\zeta}\delta(z)^{\alpha-2}\,dA(z)
 \le C_\alpha\delta(\zeta)^{\alpha-2}|E_\zeta|
 \le C_\alpha\delta(\zeta)^\alpha.
\]
A second application of Tonelli's theorem now yields
\begin{align*}
 \int_\Omega |g'(z)|^p\delta(z)^{\alpha+p}\,dA(z)
 &\le C_{p,\alpha}\int_\Omega |g(\zeta)|^p
 \left(\int_{E_\zeta}\delta(z)^{\alpha-2}\,dA(z)\right)dA(\zeta)\\
 &\le C_{p,\alpha}\int_\Omega
 |g(\zeta)|^p\delta(\zeta)^\alpha\,dA(\zeta).
\end{align*}
This proves \eqref{eq:derivative-from-function}.  Combining it with
Lemma~\ref{lem:point} gives \eqref{eq:complete-local-reverse}.
\end{proof}

The proof uses only interior disks whose radii are fixed multiples of
the boundary distance.  Thus, no
regularity of $\partial\Omega$ is involved in
Proposition~\ref{prop:local-reverse}.

\begin{remark}
For the upper half-plane $\mathbb H$, a similar argument appears in
\cite[Proposition~2.3]{MW}.  In this setting the converse estimate also
has a direct proof. The authors learnt this from a manuscript by 
T.~Liu, Y.~Shen and Y.~Yang.

Suppose that $g$ is analytic in $\mathbb H$ and
$\lim_{y\to\infty}g(x+iy)=0$.  Then
\begin{equation*}
 g(x+iy)=-i\int_y^\infty g'(x+it)\,dt
\end{equation*} 
for $x+iy\in\mathbb H$.  For $p>1$ and $\varepsilon>0$, H\"older's
inequality gives
\begin{align*}
|g(x+iy)| 
&\leq\left(\int_y^\infty
 t^{-1-\varepsilon/(p-1)}\,dt\right)^{(p-1)/p}
 \left(\int_y^\infty t^{p-1+\varepsilon}
 |g'(x+it)|^p\,dt\right)^{1/p}.
\end{align*}
Given $\alpha>-1$, choose $\varepsilon>0$ such that
$\alpha-\varepsilon>-1$.  It follows that
\[
 y^{\alpha}|g(x+iy)|^p
 \lesssim y^{\alpha-\varepsilon}\int_y^\infty
 t^{p-1+\varepsilon}|g'(x+it)|^p\,dt.
\]
Integrating over $\mathbb H$ and applying Tonelli's theorem yields
\[
 \int_{\mathbb H}|g(z)|^p y^{\alpha}\,dA(z)
 \lesssim
 \int_{\mathbb H}|g'(z)|^p y^{\alpha+p}\,dA(z).
\]
\end{remark}

\section{Weighted Poincar\'e inequalities and distance-weight integrability}
\label{sec:poincare-input}

For $1<p<\infty$ and a real exponent $\alpha$, we use the notation
\[
 \|h\|_{L^p(\Omega,\delta^\alpha)}
 :=\left(\int_\Omega |h(z)|^p\delta(z)^\alpha\,dA(z)\right)^{1/p}.
\]
All inequalities involving possibly infinite integrals are understood
in the extended sense.

The global estimate relies on the following result of
L\'opez-Garc\'\i a and Ojea \cite[Theorem~5.3]{LGO}.

\begin{theorem}[Distance-weighted Poincar\'e inequality]\label{thm:LGO}
Let $\Omega\subset\mathbb C$ be a bounded John domain, let $1<p<\infty$, and
let $\beta\in\R$ satisfy
\begin{equation}\label{eq:LGO-condition}
 \beta p> \dimA(\partial\Omega)-2.
\end{equation}
Suppose that $u\in W^{1,1}_{\mathrm{loc}}(\Omega)$ satisfies
$u\in L^p(\Omega,\delta^{\beta p})$,
$\nabla u\in L^p(\Omega,\delta^{(\beta+1)p})$,
and
\[
 \int_\Omega u(z)\delta(z)^{\beta p}\,dA(z)=0.
\]
Then there is a constant $C>0$ such that
\begin{equation}\label{eq:LGO}
 \|u\|_{L^p(\Omega,\delta^{\beta p})}
 \le C\|\nabla u\|_{L^p(\Omega,\delta^{(\beta+1)p})}.
\end{equation}
\end{theorem}

\begin{remark}
The weighted Poincar\'e inequality in
Theorem~\ref{thm:LGO} can also be
understood directly in terms of a Whitney decomposition of the
domain. Indeed, let $\mathcal W$ be a Whitney decomposition of
$\Omega$. Since
$\delta(z)\asymp \ell(Q)$ for $z\in Q$ and $Q\in\mathcal W$,
the usual local Poincar\'e inequality on each Whitney cube gives
\[
 \int_Q |u-u_Q|^p\delta(z)^{\beta p}\,dA(z)
 \lesssim
 \ell(Q)^{(\beta+1)p}
 \int_Q |\nabla u(z)|^p\,dA(z),
\]
where $u_Q$ denotes the average of $u$ over $Q$.
One then joins each $Q$ to a distinguished
Whitney cube by a chain of neighboring Whitney cubes adapted to the
John geometry and estimates the difference of the corresponding
averages by a telescoping argument. After summing over
$Q\in\mathcal W$, the bounded overlap of the enlarged cubes and the
associated discrete Hardy, or shadow-packing, estimate yield the
global weighted Poincar\'e inequality. 
In this approach, a suitable summability condition is needed to
control the contributions of Whitney cubes approaching the boundary.
Thus,
Theorem~\ref{thm:LGO} may be viewed as the global inequality obtained by
patching together the local Poincar\'e inequalities through the
Whitney geometry of a John domain.
\end{remark}

Taking $\beta=\alpha/p$ turns
\eqref{eq:LGO-condition} into \eqref{eq:general-assumptions}: $\alpha>\dimA(\partial\Omega)-2$,
and \eqref{eq:LGO} becomes
\begin{equation}\label{eq:weighted-poincare-alpha}
 \|u-u_{\Omega,\alpha}\|_{L^p(\Omega,\delta^\alpha)}
 \le \rev{C}\|\nabla u\|_{L^p(\Omega,\delta^{\alpha+p})},
\end{equation}
where
\begin{equation}\label{eq:weighted-average-alpha}
 u_{\Omega,\alpha}
 :=\frac{\displaystyle\int_\Omega u(z)\delta(z)^\alpha\,dA(z)}
 {\displaystyle\int_\Omega\delta(z)^\alpha\,dA(z)}.
\end{equation}
The weighted average is well-defined whenever
$\delta^\alpha\in L^1(\Omega)$.
This condition also follows from \eqref{eq:general-assumptions} as seen below.

For a bounded domain $\Omega\subset\mathbb C$ with boundary
$\Gamma=\partial\Omega$, set
\[
 U_t(\Gamma)=\{z\in\C:\dist(z,\Gamma)<t\},
 \qquad
 U_t^\Omega(\Gamma)= U_t(\Gamma)\cap\Omega = \{z\in\Omega:\delta(z)<t\}.
\]
The \emph{upper Minkowski dimension} and the \emph{interior upper Minkowski
dimension} are defined respectively by
\begin{align*}
 \dimM(\Gamma)
 &=\inf\left\{s\ge0:
 \limsup_{t\to 0}\frac{|U_t(\Gamma)|}{t^{2-s}}<\infty\right\},\\
 \dimMi(\Gamma)
 &=\inf\left\{s\ge0:
 \limsup_{t\to 0}\frac{|U_t^\Omega(\Gamma)|}{t^{2-s}}<\infty\right\}.
\end{align*}
Clearly $\dimMi(\Gamma)\le\dimM(\Gamma)$. It is also known that
\[
\dimH(\Gamma) \le \dimM(\Gamma) \le \dimA(\Gamma).
\]
Further, if $\Gamma$ is a rectifiable Jordan curve, then
$\dimH(\Gamma)=\dimMi(\Gamma)=\dimM(\Gamma)=1$.

The following lemma shows that the integrability of $\delta^\alpha$ for $\alpha<0$ is
characterized by $\dimMi(\Gamma)$ except at the critical exponent $\alpha=\dimMi(\Gamma)-2$. 
The case for $\alpha \neq \dimMi(\Gamma)-2$ was investigated in a more general setting by Brown \cite[Lemma~2.2]{Brown}.
For completeness, we also provide a proof for it.
In particular,
\eqref{eq:general-assumptions} implies
$\delta^\alpha\in L^1(\Omega)$ when $\alpha<0$, because
$\dimMi(\Gamma)\le\dimA(\Gamma)$; for $\alpha\ge0$, integrability follows
immediately from the boundedness of $\Omega$.

\begin{lemma}\label{lem:integrability}
Let $\Omega \subset \mathbb{C}$ be a bounded domain with boundary $\Gamma = \partial\Omega$. 
Let $\alpha < 0$.
\begin{enumerate}
\item[\rm{(i)}] If $\alpha > \dimMi(\Gamma)-2$, then $\int_\Omega \delta(z)^{\alpha}\, dA(z) < \infty$.
\item[\rm{(ii)}] If $\alpha < \dimMi(\Gamma)-2$, then $\int_\Omega \delta(z)^{\alpha}\, dA(z) = \infty$.
\item[\rm{(iii)}] If $\alpha = \dimMi(\Gamma)-2$ and
\begin{equation}\label{eq:positive-content}
\limsup_{t \to 0}\, t^{\alpha}\, |U^{\Omega}_t(\Gamma)| > 0,
\end{equation}
then $\int_\Omega \delta(z)^{\alpha}\, dA(z) = \infty$.
\end{enumerate}
In particular, if \eqref{eq:positive-content} holds, then $\delta^{\alpha} \in L^1(\Omega)$
if and only if $\alpha > \dimMi(\Gamma)-2$.
\end{lemma}

\begin{proof}
The definition of $\dimMi(\Gamma)$ implies that for any $s>\dimMi(\Gamma)$,
there exist constants $C_0>0$ and
$t_0>0$ such that
\[
 \bigl|\{z\in\Omega:\delta(z)<t\}\bigr|
 \leq C_0 t^{2-s},
 \qquad 0<t<t_0.
\]
By decomposing this boundary neighborhood into
the dyadic layers
\[
E_k=
 \left\{
 z\in\Omega:
 2^{-k-1}t_0\leq\delta(z)<2^{-k}t_0
 \right\},
 \qquad k=0,1,\ldots,
\] 
we estimate
\[
\sum_{k=0}^\infty \int_{E_k} \delta(z)^\alpha\,dA(z) \leq \sum_{k=0}^\infty C_0(2^{-k}t_0)^{2-s}(2^{-k-1}t_0)^\alpha
\asymp \sum_{k=0}^\infty2^{-(2-s+\alpha)k}.
\]
This shows that
\[
 \int_\Omega\delta(z)^\alpha\,dA(z)<\infty
\]
for every $\alpha>s-2$. Thus (i) is proved.

For (ii), we use the converse estimate
\begin{equation}\label{eq:lower-estimate}
\int_\Omega \delta(z)^{\alpha}\, dA(z)
\ \geq\ \int_{U^{\Omega}_t(\Gamma)} \delta(z)^{\alpha}\, dA(z)
\ \geq\ t^{\alpha}\, |U^{\Omega}_t(\Gamma)|, \qquad t > 0,
\end{equation}
which holds because $\delta(z) < t$ on $U^{\Omega}_t(\Gamma)$ and $\alpha < 0$.
Choose $s$ with $\alpha + 2 < s <\dimMi(\Gamma)$. The definition of
$\dimMi(\Gamma)$ gives
$\limsup_{t \to 0} t^{\,s-2}\, |U^{\Omega}_t(\Gamma)| = \infty$.
Hence there exist $t_j \to 0$ such that
$|U^{\Omega}_{t_j}(\Gamma)| \geq j\, t_j^{\,2-s}$ for every $j \in \mathbb N$.
Applying \eqref{eq:lower-estimate} with $t = t_j$, we obtain
\[
\int_\Omega \delta(z)^{\alpha}\, dA(z) \ \geq\ j\, t_j^{\,\alpha + 2 - s} \ \geq\ j
\]
for all $j$ with $t_j \leq 1$, because $\alpha + 2 - s < 0$.
Letting $j \to \infty$ proves (ii).

Finally, for (iii), 
suppose to the contrary that $\int_\Omega \delta(z)^{\alpha}\, dA(z) < \infty$. Then
\[
0 \le t^\alpha |U_t^\Omega(\Gamma)|
\le \int_{U_t^\Omega(\Gamma)}\delta(z)^\alpha\,dA
\to 0
\]
as $t \to 0$. However, this contradicts \eqref{eq:positive-content}, which proves (iii).
\end{proof}

\section{\texorpdfstring{The global direction:
from higher to lower derivatives}{The global direction: from higher to
lower derivatives}}
\label{sec:base-disk}

Throughout this section, assume that $\Omega$ is a bounded John
domain, $1<p<\infty$, and
$\alpha>\dimA(\partial\Omega)-2$, as in
\eqref{eq:general-assumptions}.  Fix $z_0\in\Omega$ and write
\[
 \delta_0=\delta(z_0),
 \qquad
 B_0=B\left(z_0,\frac{\delta_0}{4}\right),
 \qquad
 M_\alpha=\left(\int_\Omega\delta(z)^\alpha\,dA(z)\right)^{1/p}<\infty.
\]
For a locally integrable real-valued function $u$, put
\[
 u_{B_0}=\frac{1}{|B_0|}\int_{B_0}u(z)\,dA(z).
\]

\begin{lemma}\label{lem:bounded-base}
Let $u$ be a bounded real-valued function in
$W^{1,1}_{\mathrm{loc}}(\Omega)$ such that
$\nabla u\in L^p(\Omega,\delta^{\alpha+p})$.  Then
\begin{equation}\label{eq:base-disk-poincare}
 \|u-u_{B_0}\|_{L^p(\Omega,\delta^\alpha)}
 \le C_0\|\nabla u\|_{L^p(\Omega,\delta^{\alpha+p})},
\end{equation}
where $C_0$ is independent of $u$.
\end{lemma}

\begin{proof}
Since $u$ is bounded and $M_\alpha<\infty$, we have
$u\in L^p(\Omega,\delta^\alpha)$.  Let
$a=u_{\Omega,\alpha}$ be the weighted average in
\eqref{eq:weighted-average-alpha}.  The weighted Poincar\'e inequality
\eqref{eq:weighted-poincare-alpha} gives
\begin{equation}\label{eq:ua-poincare}
 \|u-a\|_{L^p(\Omega,\delta^\alpha)}
 \le C_p\|\nabla u\|_{L^p(\Omega,\delta^{\alpha+p})}.
\end{equation}
Moreover,
\[
 |u_{B_0}-a|
 \le |B_0|^{-1/p}\|u-a\|_{L^p(B_0)}.
\]
For $z\in B_0$, we have
$3\delta_0/4\le\delta(z)\le 5\delta_0/4$,
so $\delta^\alpha$ is comparable to $\delta_0^\alpha$ on $B_0$.
Since $|B_0|\asymp\delta_0^2$, it follows that
\begin{equation}\label{eq:average-comparison}
 |u_{B_0}-a|
 \le C_{p,\alpha}\delta_0^{-(\alpha+2)/p}
 \|u-a\|_{L^p(\Omega,\delta^\alpha)}.
\end{equation}
Using Minkowski's inequality, \eqref{eq:average-comparison}, and
\eqref{eq:ua-poincare}, we obtain
\begin{align*}
 \|u-u_{B_0}\|_{L^p(\Omega,\delta^\alpha)}
 &\le \|u-a\|_{L^p(\Omega,\delta^\alpha)}
      +M_\alpha|a-u_{B_0}|\\
 &\le \left(1+C_{p,\alpha}M_\alpha
      \delta_0^{-(\alpha+2)/p}\right)
      \|u-a\|_{L^p(\Omega,\delta^\alpha)}\\
 &\le C_0\|\nabla u\|_{L^p(\Omega,\delta^{\alpha+p})}.
\end{align*}
This proves the statement.
\end{proof}

The boundedness assumption can be removed for harmonic functions by
truncation.

\begin{lemma}\label{lem:harmonic-base}
If $u$ is a real-valued harmonic function in $\Omega$, then
\begin{equation}\label{eq:harmonic-base}
 \|u-u_{B_0}\|_{L^p(\Omega,\delta^\alpha)}
 \le C_0\|\nabla u\|_{L^p(\Omega,\delta^{\alpha+p})}.
\end{equation}
\end{lemma}

\begin{proof}
If the right-hand side of \eqref{eq:harmonic-base} is infinite, there is nothing
to prove.  We may therefore assume that
$\nabla u\in L^p(\Omega,\delta^{\alpha+p})$.
For $N>0$, define the truncation
\[
 T_N(t)=
 \begin{cases}
  -N, & t<-N,\\
  t,  & -N\le t\le N,\\
  N,  & t>N,
 \end{cases}
 \qquad
 u_N=T_N\circ u.
\]
The function $T_N$ is $1$-Lipschitz.  Since $u$ is smooth, the Sobolev
chain rule gives
\begin{equation*}
 |\nabla u_N|\le|\nabla u|
 \qquad\text{a.e. in }\Omega.
\end{equation*}
Applying Lemma~\ref{lem:bounded-base} to $u_N$ yields
\begin{equation}\label{eq:truncated-estimate}
 \|u_N-(u_N)_{B_0}\|_{L^p(\Omega,\delta^\alpha)}
 \le C_0\|\nabla u\|_{L^p(\Omega,\delta^{\alpha+p})}.
\end{equation}
Because $\overline{B_0}\Subset\Omega$, the harmonic function $u$ is
bounded on $B_0$.  Thus, for all sufficiently large $N$,
$u_N=u$ on $B_0$ and consequently $(u_N)_{B_0}=u_{B_0}$.  Also,
$u_N(z)\to u(z)$ pointwise in $\Omega$.  Fatou's lemma applied to
\eqref{eq:truncated-estimate} proves \eqref{eq:harmonic-base}.
\end{proof}

\begin{proposition}
\label{prop:global-analytic}
Under the assumptions of Theorem~\ref{thm:main-general}, every analytic
function $g$ in $\Omega$ satisfies
\begin{equation}\label{eq:function-from-derivative}
 \int_\Omega |g(z)|^p\delta(z)^\alpha\,dA(z)
 \le C\left\{
 \int_\Omega |g'(z)|^p\delta(z)^{\alpha+p}\,dA(z)
 +|g(z_0)|^p\delta_0^{\alpha+2}\right\}.
\end{equation}
\end{proposition}

\begin{proof}
We may assume that the right-hand side of
\eqref{eq:function-from-derivative} is finite. 
Write $g=u+iv$, where $u$ and $v$ are real-valued harmonic functions.
The Cauchy--Riemann equations imply
\begin{equation}\label{eq:CR-gradients}
 |\nabla u|=|\nabla v|=|g'|.
\end{equation}
By the mean-value property on the disk $B_0$, we have
$u_{B_0}=u(z_0)$ and $v_{B_0}=v(z_0)$.
Lemma~\ref{lem:harmonic-base} and \eqref{eq:CR-gradients} therefore give
\begin{align*}
 \|g-g(z_0)\|_{L^p(\Omega,\delta^\alpha)}
 &\le \|u-u(z_0)\|_{L^p(\Omega,\delta^\alpha)}
      +\|v-v(z_0)\|_{L^p(\Omega,\delta^\alpha)}\\
 &\le 2C_0\|g'\|_{L^p(\Omega,\delta^{\alpha+p})}.
\end{align*}
It follows that
\begin{align*}
 \|g\|_{L^p(\Omega,\delta^\alpha)}
 &\le 2C_0\|g'\|_{L^p(\Omega,\delta^{\alpha+p})}
      +M_\alpha|g(z_0)|\\
 &\le C\left(
 \|g'\|_{L^p(\Omega,\delta^{\alpha+p})}
 +\delta_0^{(\alpha+2)/p}|g(z_0)|\right),
\end{align*}
where the constant absorbs
$M_\alpha\delta_0^{-(\alpha+2)/p}$.  Raising this inequality to the
$p$-th power proves \eqref{eq:function-from-derivative}.
\end{proof}

\begin{remark}[The shorter argument under an a priori integrability assumption]
Suppose in addition that $g\in L^p(\Omega,\delta^\alpha)$, and let
$b=g_{\Omega,\alpha}$.  Applying
\eqref{eq:weighted-poincare-alpha} to the real and imaginary parts
gives
\[
 \|g-b\|_{L^p(\Omega,\delta^\alpha)}
 \lesssim \|g'\|_{L^p(\Omega,\delta^{\alpha+p})}.
\]
Since $|g-b|^p$ is subharmonic, the mean-value inequality on $B_0$
gives
\[
 |g(z_0)-b|
 \lesssim\delta_0^{-(\alpha+2)/p}
 \|g-b\|_{L^p(\Omega,\delta^\alpha)}.
\]
Thus
\[
 \|g\|_{L^p(\Omega,\delta^\alpha)}
 \lesssim
 \|g'\|_{L^p(\Omega,\delta^{\alpha+p})}
 +\delta_0^{(\alpha+2)/p}|g(z_0)|.
\]
The truncation argument above is what removes the a priori assumption
that the left-hand side is finite.
\end{remark}

\section{Sharpness for \texorpdfstring{$1<p<2$}{1<p<2}: quasidisk counterexamples}
\label{sec:counterexamples}

We now show that the dimension restriction in the first-derivative comparison is essential. 
More precisely, when $1<p<2$, the comparison may fail on bounded John domains whose boundary is sufficiently large.

For a bounded John domain $\Omega$ with boundary $\Gamma$,
the upper Minkowski dimension $\dimM(\Gamma)$ and the interior upper Minkowski dimension $\dimMi(\Gamma)$
are the same.  We include the short argument because it identifies
precisely where the John geometry enters.

\begin{lemma}
\label{lem:minkowski-corkscrew}
Let $\Omega\subset\C$ be a bounded domain satisfying the interior
corkscrew condition; that is, there exist $c\in(0,1)$ and $r_0>0$
such that, for every $\xi\in\partial\Omega$ and $0<r<r_0$, there is a
point $x\in\Omega\cap B(\xi,r)$ for which
$B(x,cr)\subset\Omega\cap B(\xi,r)$.
Then
\[
 \dimMi(\partial\Omega)=\dimM(\partial\Omega).
\]
In particular, this equality holds for every bounded John domain.
\end{lemma}

\begin{proof}
Write $\Gamma=\partial\Omega$.  Fix
$0<r<\min\{r_0,\diam\Omega\}$ and choose a maximal
$4r$-separated family $\{\xi_j\}_{j=1}^N\subset\Gamma$.  Maximality
gives
\begin{equation}\label{eq:covering}
 \Gamma\subset\bigcup_{j=1}^N B(\xi_j,4r),
\end{equation}
and hence
\begin{equation}\label{eq:outer-layer-cover}
 |U_r(\Gamma)|\le 25N \pi r^2.
\end{equation}
For each $j$, choose an interior corkscrew disk
$B(x_j,cr)\subset\Omega\cap B(\xi_j,r)$.
The centers satisfy
\[
 |x_i-x_j|\ge |\xi_i-\xi_j|-|x_i-\xi_i|-|x_j-\xi_j|
 \ge 2r
 \qquad (i\ne j),
\]
so these corkscrew disks are pairwise disjoint.  Moreover, if
$y\in B(x_j,cr) \subset B(\xi_j,r)$, then
\[
 \delta(y)\le |y-\xi_j|<r.
\]
Therefore
\begin{equation}\label{eq:inner-layer-lower}
|U_r^\Omega(\Gamma)|\ge N\pi c^2r^2.
\end{equation}
Combining \eqref{eq:outer-layer-cover} and
\eqref{eq:inner-layer-lower} yields
$|U_r(\Gamma)|\le 25|U_r^\Omega(\Gamma)|/c^2$.
Together with the obvious inequality
$|U_r^\Omega(\Gamma)|\le|U_r(\Gamma)|$, this
proves the equality of the two upper Minkowski dimensions.  

Recall that by \cite[Lemma~8.9]{ABBS} every bounded John domain satisfies the
interior corkscrew condition. Thus the asserted equality holds for every
bounded John domain.
\end{proof}

\begin{proposition}
\label{prop:dimension-counterexample}
Let $1<p<2$, and let $\Omega$ be a bounded John domain such that either
$\dimM(\partial\Omega)> p$ or
$\dimM(\partial\Omega)=p$ and 
$$
\limsup_{t\to 0}t^{\,p-2}|U^\Omega_t(\partial\Omega)|>0.
$$
For any fixed
$z_0\in\Omega$, there is no finite constant $C$ such that
\begin{equation}\label{eq:false-global-comparison}
 \int_\Omega |g(z)|^p\delta(z)^{p-2}\,dA(z)
 \le C\left\{
  |g(z_0)|^p\delta(z_0)^p
  +\int_\Omega |g'(z)|^p\delta(z)^{2p-2}\,dA(z)
 \right\}
\end{equation}
for every analytic function $g$ in $\Omega$.
\end{proposition}

\begin{proof}
Take $g\equiv1$.  Its derivative vanishes, and the right-hand side of
\eqref{eq:false-global-comparison} is the finite number
$C\delta(z_0)^p$.  
The left-hand side is infinite. Indeed, Lemma \ref{lem:minkowski-corkscrew} gives $\dimMi(\partial\Omega)=\dimM(\partial\Omega)$, so that Lemma \ref{lem:integrability} (ii) or (iii) 
applies with $\alpha=p-2$.
\end{proof}

In particular, the
comparison in Theorem~\ref{thm:main-general} cannot hold on all
bounded John domains when $1<p<2$, because bounded quasidisks can have
boundary dimension at least $p$.  We now give an explicit family of such
counterexamples.

The snowflake construction by Rohde \cite{Rohde} produces quasicircles by repeatedly
replacing each line segment by a four-segment polygonal arc.
Fix a parameter
$\rho\in(1/4,1/2)$
and use the $\rho$-replacement at every stage (Figure~\ref{fig:rohde-iteration} illustrates the iterative replacement): each segment is replaced by four segments, each of length $\rho$ times the length of its parent segment, and the same operation is performed on every side of the initial square with a side length $1$. The resulting
self-similar Rohde snowflake $\Gamma_\rho$ is a quasicircle.  

\begin{figure}[htbp]
\centering
\begin{tikzpicture}[line cap=round,line join=round,
  every node/.style ={text=black},>={Stealth[length=2.2mm]}]
  \begin{scope}[xshift=0cm,yshift=0.65cm,x=2.65cm,y=2.65cm]
    \draw[line width=0.9pt] (0,0)--(1,0);
  \end{scope}
  \node at (1.325,0.25) {$I_0$};

  \draw[->,line width=0.65pt] (2.80,0.65)--(3.35,0.65);

  \begin{scope}[xshift=3.50cm,yshift=0.65cm,x=2.65cm,y=2.65cm]
    \draw[line width=0.9pt]
      (0,0)--(0.34,0)--(0.50,0.30)--(0.66,0)--(1,0);
    \node[font=\scriptsize,above] at (0.17,0) {$\rho$};
    \node[font=\scriptsize,above left] at (0.42,0.15) {$\rho$};
    \node[font=\scriptsize,above right] at (0.58,0.15) {$\rho$};
    \node[font=\scriptsize,above] at (0.83,0) {$\rho$};
  \end{scope}
  \node at (4.825,0.25) {$I_1$};

  \draw[->,line width=0.65pt] (6.30,0.65)--(6.85,0.65);

  \begin{scope}[xshift=7.00cm,yshift=0.65cm,x=2.65cm,y=2.65cm]
    \draw[line width=0.9pt]
      (0.00000,0.00000)--(0.11560,0.00000)--(0.17000,0.10200)--
      (0.22440,0.00000)--(0.34000,0.00000)--(0.39440,0.10200)--
      (0.33000,0.19800)--(0.44560,0.19800)--(0.50000,0.30000)--
      (0.55440,0.19800)--(0.67000,0.19800)--(0.60560,0.10200)--
      (0.66000,0.00000)--(0.77560,0.00000)--(0.83000,0.10200)--
      (0.88440,0.00000)--(1.00000,0.00000);
  \end{scope}
  \node at (8.325,0.25) {$I_2$};

  \node[font=\Large] at (10.15,0.66) {$\cdots$};
  \draw[->,line width=0.65pt] (10.48,0.65)--(11.05,0.65);
  \node[align=center] at (12.05,0.66)
    {$I_\infty$\\[-1mm]{\scriptsize(one of four sides of $\Gamma_\rho$)}};
\end{tikzpicture}
\caption{Iterative generation of a self-similar
Rohde snowflake arc.}
\label{fig:rohde-iteration}
\end{figure}
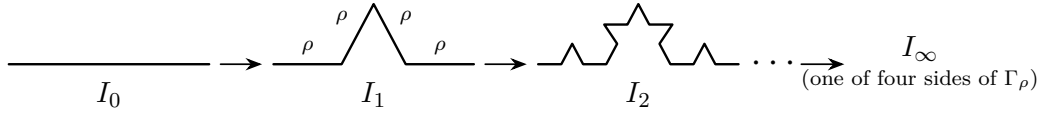

Each side of the initial polygon generates a self-similar arc
consisting of four similar copies, each with contraction ratio
$\rho$.  These arcs satisfy the open set condition.  Hence each of
them, and therefore their finite union $\Gamma_\rho$, has Hausdorff
dimension $d$ determined by
$4\rho^d=1$.
Moreover, $\Gamma_\rho$ is $d$-Ahlfors regular, and consequently
$\dimH(\Gamma_\rho)=d$.

\begin{corollary}
\label{cor:rohde-counterexamples}
For every $1<p<2$ and every $d\in[p,2)$, there exists a bounded
quasidisk $\Omega$ such that
\[
 \dimH(\partial\Omega)=\dimM(\partial\Omega)=\dimA(\partial\Omega)=d
 \qquad\text{and}\qquad
 \int_{\Omega}\delta(z)^{p-2}\,dA(z)=\infty.
\]
Consequently, the constant function $g\equiv1$ is a counterexample to
\eqref{eq:false-global-comparison} on $\Omega$.
\end{corollary}

\begin{proof}
Choose
$\rho=4^{-1/d}$.
Because $1<d<2$, one has $1/4<\rho<1/2$.  Let $\Gamma_\rho$ be the
self-similar Rohde snowflake described above, and let $\Omega$ be its
bounded complementary component.  Then $\Omega$ is a quasidisk and
\[
\dimH(\partial\Omega)=\dimM(\partial\Omega)
=\dimA(\partial\Omega)=d\ge p.
\] 
Here the equality of all three dimensions follows from the
$d$-Ahlfors regularity noted above.  

Since $\Gamma_\rho$ is $d$-Ahlfors regular,  
there exist $c_1>0$ and $r_1>0$ such that
\begin{equation}\label{eq:new}
|U^\Omega_r(\partial\Omega)|\ \ge\ c_1 r^{2-d} 
\end{equation}
for every $r \in (0,r_1)$. Indeed,
by \eqref{eq:covering}, we have 
\[
\mathcal H^d(\Gamma)\le\sum_j\mathcal H^d(\Gamma\cap B(\xi_j,4r))\le NA(4r)^d
\] 
for some $A>0$ whenever $4r\le\operatorname{diam}\Gamma$. 
Hence $N\ge\mathcal H^d(\Gamma)/(A4^dr^d)$, and \eqref{eq:inner-layer-lower} yields \eqref{eq:new} 
with $c_1=\pi c^2\mathcal H^d(\Gamma)/(A4^d)$.

Hence the hypothesis of Proposition \ref{prop:dimension-counterexample} is satisfied: 
if $d>p$, then $\dim_M(\partial\Omega)=d>p$; if $d=p$, 
then $\limsup_{t\to 0}t^{\,p-2}|U^\Omega_t(\partial\Omega)|\ge c_1>0$ by \eqref{eq:new}.
Thus
Proposition~\ref{prop:dimension-counterexample} gives the
counterexample.
In particular, taking $d=p$ produces a bounded quasidisk whose boundary
has Hausdorff dimension exactly $p$.
\end{proof}

\section{\texorpdfstring{An inward-cusp
counterexample}{An inward-cusp counterexample}}
\label{sec:cusp-counterexample}

For higher derivatives, the boundary-dimension restriction in
Corollary~\ref{cor:delicate} disappears when $m\ge2$.  One may therefore
ask whether the John-domain hypothesis can also be weakened.  The
following example shows that it cannot in general: the same construction
works for every $m\ge1$ and every prescribed $s>1$.

\begin{proposition}
\label{prop:cusp-counterexample}
Let $1<p<\infty$, let $m\ge1$ be an integer, and let $s>1$.
There exist a bounded, simply connected Jordan domain $\Omega\subset\C$
and an analytic function $f$ on $\Omega$ such that $\Omega$ has
rectifiable boundary, is an $s$-John domain but not a John domain, and,
for every fixed $z_0\in\Omega$,
\[
 \int_\Omega |f^{(m)}(z)|^p\delta(z)^{mp-2}\,dA(z)=\infty,
\]
whereas
\[
 |f^{(m)}(z_0)|^p\delta(z_0)^{mp}
 +
 \int_\Omega |f^{(m+1)}(z)|^p
 \delta(z)^{(m+1)p-2}\,dA(z)<\infty.
\]
Consequently, the John-domain assumption in the higher-derivative
comparison cannot in general be omitted, even when $m\geq2$.
\end{proposition}

\begin{proof}
Consider the inward cusp domain
\[
 \Omega_s
 =
 \{z=x+iy:0<x<1,\ |y|<x^s\}.
\]
See Figure~\ref{fig:inward-cusp}.
This is a bounded simply connected Jordan domain with rectifiable
boundary.  It is not a John domain.  Indeed, if
$z=x+iy\in\Omega_s\cap B(0,r)$, then
\[
 \delta(z)\leq x^s-|y|\leq x^s\leq r^s.
\]
Hence every disk contained in $\Omega_s\cap B(0,r)$ has radius at
most $r^s$.  Since $s>1$, this rules out an interior corkscrew disk
of radius comparable to $r$ at the cusp point $0$.
Since every bounded John domain satisfies the interior corkscrew
condition (see \cite[Lemma~8.9]{ABBS}), $\Omega_s$ cannot be a John domain.

\begin{figure}[htbp]
\centering
\begin{tikzpicture}[x=6.2cm,y=1.35cm,
  every node/.style={text=black},
  >={Stealth[length=2.2mm]}]
  \fill[revisionblue!13]
    plot[domain=0:1,samples=100] (\x,{\x^2})
    -- (1,-1)
    -- plot[domain=1:0,samples=100] (\x,{-\x^2}) -- cycle;
  \draw[line width=0.9pt]
    plot[domain=0:1,samples=100] (\x,{\x^2});
  \draw[line width=0.9pt]
    plot[domain=0:1,samples=100] (\x,{-\x^2});
  \draw[line width=0.9pt] (1,-1)--(1,1);
  \draw[->,line width=0.55pt] (-0.06,0)--(1.12,0)
    node[below] {$x$};
  \draw[->,line width=0.55pt] (0,-1.20)--(0,1.25)
    node[left] {$y$};
  \node[above] at (0.58,0.45) {$y=x^s$};
  \node[below] at (0.58,-0.45) {$y=-x^s$};
  \node[above right] at (0.02,0.02) {inward cusp};
  \node[below] at (1,0) {$1$};
  \node at (0.72,0) {$\Omega_s$};
\end{tikzpicture}
\caption{The inward cusp domain
$\Omega_s=\{x+iy:0<x<1,\ |y|<x^s\}$ for $s>1$.}
\label{fig:inward-cusp}
\end{figure}
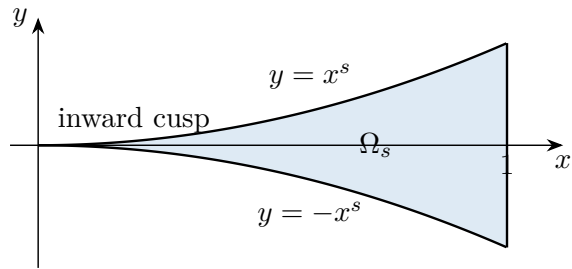

We now verify the assertion that $\Omega_s$ is an
$s$-John domain, using the definition given in the introduction.
Set $a=1/4$ and take $z_*=(a,0)$ as the John center.  First consider
$z=(x,y)\in\Omega_s$ with $0<x\leq a$.  Join $z$ vertically to
$(x,0)$ and then horizontally to $z_*$.  Parametrize this polygonal
arc by arclength $t$ starting at $z$.  On the vertical part, suppose
first that $y\ge0$ and write $\gamma(t)=(x,y-t)$ for $0\le t\le y$.
The vertical gap from $\gamma(t)$ to the upper graph
$$
\{(x, x^s):\; 0\leq x\leq 1\}
$$
is
$x^s-y+t\ge t$. Moreover, this graph is a Lipschitz graph with Lipschitz constant at most $s$. The standard distance estimate for a Lipschitz graph yields that the distance to a Lipschitz graph is comparable to the vertical distance to that graph. Therefore, 
\[
  \delta(\gamma(t))\geq c_s t\geq c_s t^s,
\]
because this part has length at most $x^s\leq1$. The case $y<0$ is identical using the lower graph.

On the horizontal part, write $\gamma(t)=(\xi,0)$, where
$x\leq\xi\leq a$.  The same Lipschitz-graph estimate gives
$\delta(\xi,0)\geq c_s\xi^s$.  Moreover, the arclength from $z$ to
$(\xi,0)$ satisfies
\[
 t=|y|+(\xi-x)\leq x^s+\xi-x\leq\xi,
\]
since $0<x<1$ and hence $x^s\leq x$.  Therefore
\[
 \delta(\gamma(t))\geq c_s\xi^s\geq c_s t^s.
\]

It remains to treat points with $x>a$.  The truncated domain
\[
 D=\Omega_s\cap\{x>a/2\}
\]
is a bounded Lipschitz domain, hence a uniform domain and, in particular,
a John domain.  Changing its John center to the fixed interior point
$z_*$ only changes the John constant.  Thus there are constants $c>0$
and $L<\infty$, independent of $z\in D$, such that $z$ can be joined to
$z_*$ by an arclength-parametrized curve of length at most $L$ satisfying
$\delta_D(\gamma(t))\geq c t$.  Since
$\delta_{\Omega_s}\geq\delta_D$ and
$t\geq L^{1-s}t^s$ for $0\leq t\leq L$, this curve satisfies
\[
 \delta_{\Omega_s}(\gamma(t))\geq cL^{1-s}t^s.
\]
Combining the two cases proves that $\Omega_s$ is an $s$-John domain.

Put
$\alpha=mp-2$
and choose
\[
 q=\frac{s(\alpha+1)+1}{p}
   =\frac{s(mp-1)+1}{p}>0.
\]
Since $\Omega_s$ is contained in the right half-plane, we may use
the principal branch of the logarithm and define
$g(z)=z^{-q}$ for 
$z\in\Omega_s$.
Then $g$ is analytic on $\Omega_s$ and
$g'(z)=-qz^{-q-1}$.

We first show that
\[
 \int_{\Omega_s}|g(z)|^p\delta(z)^\alpha\,dA(z)=\infty.
\]
Since $s>1$, we have $|z|\asymp x$ throughout $\Omega_s$.  Moreover,
for sufficiently small $\varepsilon>0$,
$\delta(x+iy)\asymp x^s$
on the central part of the cusp
\[
 E_\varepsilon
 =
 \left\{
 x+iy:
 0<x<\varepsilon,\quad |y|<\frac12x^s
 \right\}.
\]
This can also be derived from the comparability between the distance to a Lipschitz graph and  the vertical distance to that graph.  

It follows that
\begin{equation*}
 \int_{\Omega_s}|g(z)|^p\delta(z)^\alpha\,dA(z)
 \gtrsim
 \int_0^\varepsilon
 \int_{-x^s/2}^{x^s/2}
 x^{-qp}x^{s\alpha}\,dy\,dx 
\asymp
 \int_0^\varepsilon
 x^{-qp+s(\alpha+1)}\,dx.
\end{equation*}
By the definition of $q$,
$-qp+s(\alpha+1)=-1$,
and therefore
\[
 \int_{\Omega_s}|g(z)|^p\delta(z)^\alpha\,dA(z)
 =\infty.
\]

We next estimate the derivative term.  Since
$\delta(x+iy)\leq x^s-|y|$
and
$\alpha+p=(m+1)p-2>0$,
we obtain
\begin{align*}
 \int_{\Omega_s}|g'(z)|^p
 \delta(z)^{\alpha+p}\,dA(z)
 &\lesssim
 \int_0^1 x^{-(q+1)p}
 \int_{-x^s}^{x^s}
 (x^s-|y|)^{\alpha+p}\,dy\,dx \\
 &\asymp
 \int_0^1
 x^{-(q+1)p+s(\alpha+p+1)}\,dx.
\end{align*}
Using
$qp=s(\alpha+1)+1$,
the exponent in the last integral becomes
\begin{equation*}
 -(q+1)p+s(\alpha+p+1)
 =-1+p(s-1).
\end{equation*}
Since $s>1$, we have
$-1+p(s-1)>-1$,
and hence
\[
 \int_{\Omega_s}|g'(z)|^p
 \delta(z)^{\alpha+p}\,dA(z)<\infty.
\]

Finally, since $\Omega_s$ is simply connected, the analytic function
$g$ admits an $m$-fold primitive.  Thus there exists an analytic
function $f$ on $\Omega_s$ such that
$f^{(m)}=g$.
Consequently,
$f^{(m+1)}=g'$,
and the preceding estimates give
\[
 \int_{\Omega_s}|f^{(m)}(z)|^p
 \delta(z)^{mp-2}\,dA(z)=\infty
\]
while
\[
 \int_{\Omega_s}|f^{(m+1)}(z)|^p
 \delta(z)^{(m+1)p-2}\,dA(z)<\infty.
\]
For every fixed $z_0\in\Omega_s$, the quantity
\[
 |f^{(m)}(z_0)|^p\delta(z_0)^{mp}
 =
 |g(z_0)|^p\delta(z_0)^{mp}
\]
is finite.  Hence the global comparison in
Corollary~\ref{cor:consecutive} fails on $\Omega_s$.
\end{proof}

\section*{Statements and Declarations}
\noindent\textbf{Conflict of interest.} On behalf of all authors, the corresponding author states that there is no conflict of interest regarding the publication of this paper. 

\smallskip
\noindent\textbf{Data Availability.} This paper has no associated data.

\end{document}